\documentclass{amsart}     
\usepackage{amssymb} 
\usepackage{amsmath,amscd,eucal,latexsym} 
\usepackage{graphicx}

\newtheorem{theorem}{Theorem}[section]
\newtheorem{lemma}[theorem]{Lemma}
\newtheorem{corollary}[theorem]{Corollary}
\newtheorem{proposition}[theorem]{Proposition}

\theoremstyle{definition} 
\newtheorem{remark}{Remark}

\newtheorem{example}[theorem]{Example}
\newtheorem{definition}[theorem]{Definition}

\DeclareMathOperator{\intr}{int}

\DeclareMathOperator{\id}{id}

\newcommand{\wt}{\widetilde} 
\renewcommand{\tilde}{\wt}

\renewcommand{\tilde}{\wt}

\renewcommand{\phi}{\varphi}

\renewcommand{\epsilon}{\varepsilon}

\DeclareMathOperator{\area}{area}

\begin{document}

\title{Hyperbolic Geometry and Homotopic  Homeomorphisms of Surfaces}

\author[J. Cantwell]{John Cantwell}
\address{Department of Mathematics\\ St. Louis University\\ St. 
Louis, MO 
63103}
\email{CANTWELLJC@SLU.EDU}

\author[L. Conlon]{Lawrence Conlon}
\address{Department of Mathematics\\ Washington University, St. 
Louis, MO 
63130}
\email{LC@MATH.WUSTL.EDU}


\begin{abstract}
The Epstein-Baer theory of curve isotopies is basic to the remarkable theorem that homotopic homeomorphisms of surfaces are isotopic. The  groundbreaking work of R.~Baer was carried out on closed, orientable surfaces and extended by 
D.~B.~A.~Epstein to arbitrary surfaces, compact or not, with or without boundary and orientable or not.  We give a new method of deducing the theorem about homotopic homeomorphisms from the results about homotopic curves via  the  hyperbolic geometry  of surfaces.  This works on all but 13 surfaces where ad hoc proofs are needed.

\end{abstract}

\maketitle

\section{Introduction}
On arbitrary connected surfaces $L$, compact or not, orientable or not and with or without boundary,  D.~B.~A.~Epstein proves  that (with some qualifications)  properly homotopic homeomorphisms of surfaces $L$ are isotopic~\cite{Epstein:isotopy}. For this, he extends a result of R.~Baer~\cite{baer} to prove that homotopic, essential, simple, 2-sided closed curves on $L$ are ambient isotopic.  To accomodate the case that $\partial L\ne\emptyset$, he further proves that properly imbedded arcs which are homotopic mod the endpoints are ambient isotopic mod the endpoints.  These are Theorems~2.1 and~3.1 in~\cite{Epstein:isotopy}, the proofs of which in the PL category are quite elegant.  Their truth in the topological category follows from results in~\cite[Appendix]{Epstein:isotopy}  about approximating topological curves and homeomorphisms by their PL counterparts in dimension two.  We restate these results as Theorem~\ref{2.1} and Theorem~\ref{3.1} and refer the reader to~\cite{Epstein:isotopy} for the proofs.

The method in~\cite{Epstein:isotopy} of deducing  from these two  results  the theorem about homotopic homeomorphisms is a bit involved and requires that the homotopy be proper if $\partial L$ has any noncompact components.  In this note we propose a new method of deducing the theorem (in fact, a stronger theorem) from the same two results using a hyperbolic metric on $L$  with certain reasonable properties  (which we will call a \emph{standard} metric).   There are exactly $13$ exceptional surfaces, which do not admit a standard metric  (see Section~\ref{13}).   Our proof, therefore, works for all surfaces except the ``unlucky'' $13$ and nowhere requires the homotopy to be proper.

In the following definition and throughout the paper we use the open unit disk $\Delta$ as our model of the hyperbolic plane and denote the circle at infinity by $ S^{1}_{\infty}$.

\begin{definition}\label{1/2plane}
 A  \emph{hyperbolic half plane} $H$ is a hyperbolic surface isometric to the subsurface of  $\Delta$ which is the  union of a geodesic $\gamma\subset\Delta$   and one of the components of $\Delta\smallsetminus \gamma$.
\end{definition}

\begin{remark}
Generally, we just refer to these as half planes. Many results   fail for hyperbolic metrics with imbedded half planes on non-compact surfaces. Any non-compact surface can be given a complete hyperbolic metric with geodesic boundary and imbedded half planes (see the example below). 
\end{remark}

\begin{definition}
A  hyperbolic metric on a surface $L$ is called ``standard'' if  it is complete, makes $\partial L$ geodesic, and admits no isometrically imbedded half planes.  We call a surface equipped with such a standard metric a standard hyperbolic surface.  If a surface is homeomorphic to a standard hyperbolic surface, it will simply be called a standard surface.
\end{definition}

In this paper we prove for standard hyperbolic surfaces  several results that are well-known for complete, hyperbolic surfaces with geodesic boundary and finite area.  We believe that our proofs of two of these results (Theorems~\ref{extcont} and~\ref{isoisota}) are new, even for compact surfaces.

\begin{example}
We give an easy example showing how half planes can occur in a noncompact surface with infinite genus. In Figure~\ref{semipl} we depict a 2-ended surface $L$ of infinite genus and a shaded subsurface $H'$ homeomorphic to $\mathbb{R}\times[0,\infty)$.  Excise the interior of $H'$ and on the remaining surface  put  a complete hyperbolic metric making the one boundary component a geodesic.  Now glue on a hyperbolic half plane by an isometry along the boundary to obtain a complete hyperbolic metric on $L$ with an isometrically imbedded hyperbolic half plane $H$, as pictured in Figure~\ref{hypsemipl}.
\end{example}

Throughout this paper we assume that every hyperbolic surface (standard or not) is complete with geodesic boundary.
We will use the term ``half plane'' for ``isometrically imbedded hyperbolic half plane''.

\begin{figure}
\begin{center}
\begin{picture}(300,200)(-95,0)

%\epsfxsize=100pt
%\epsffile{Fig1.eps}
\includegraphics[width=100pt]{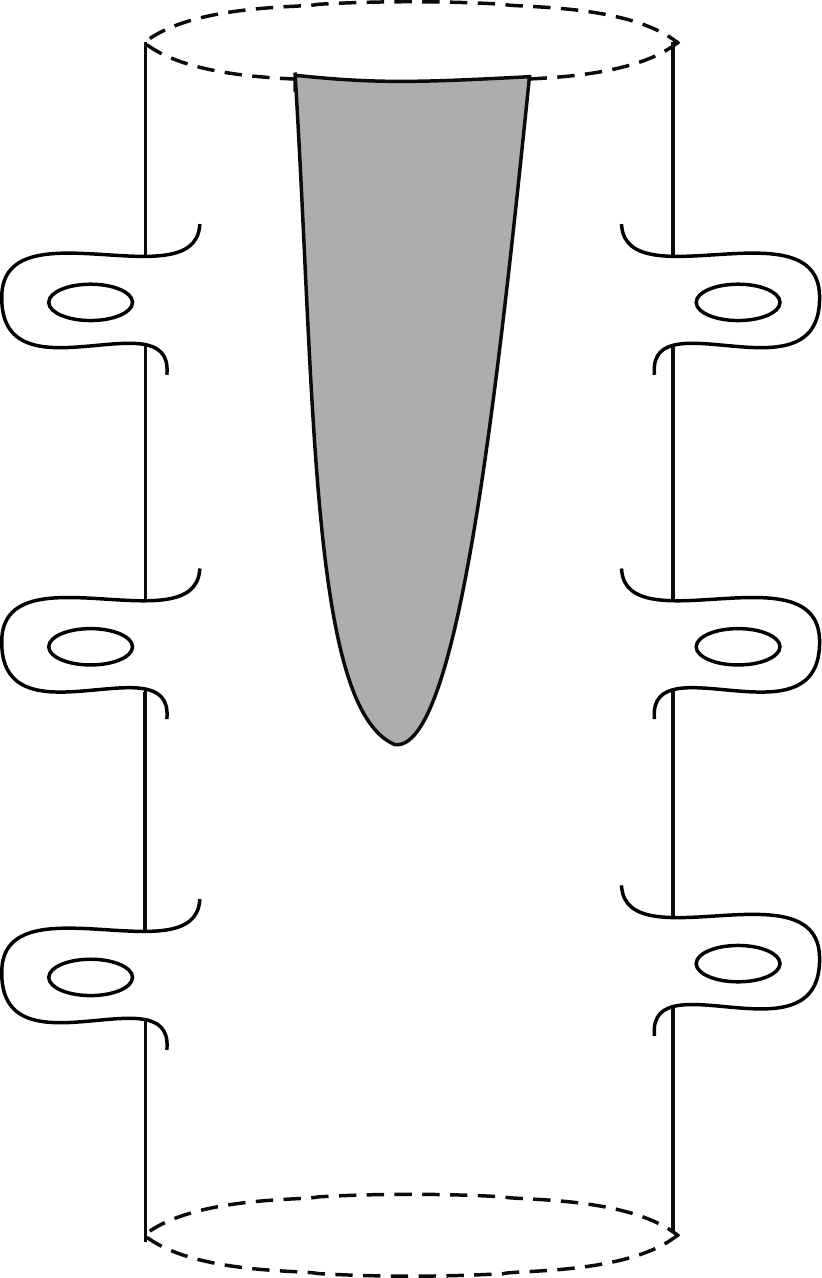}
\put(-55,130){\small$H'$}
\end{picture}
\caption{A narrow imbedded half plane $H'\cong\mathbb{R}\times[0,\infty)$}\label{semipl}
\end{center}
\end{figure}

\begin{figure}
\begin{center}
\begin{picture}(300,200)(-60,-10)

%\epsfxsize=160pt
%\epsffile{Fig2.eps}
\includegraphics[width=160pt]{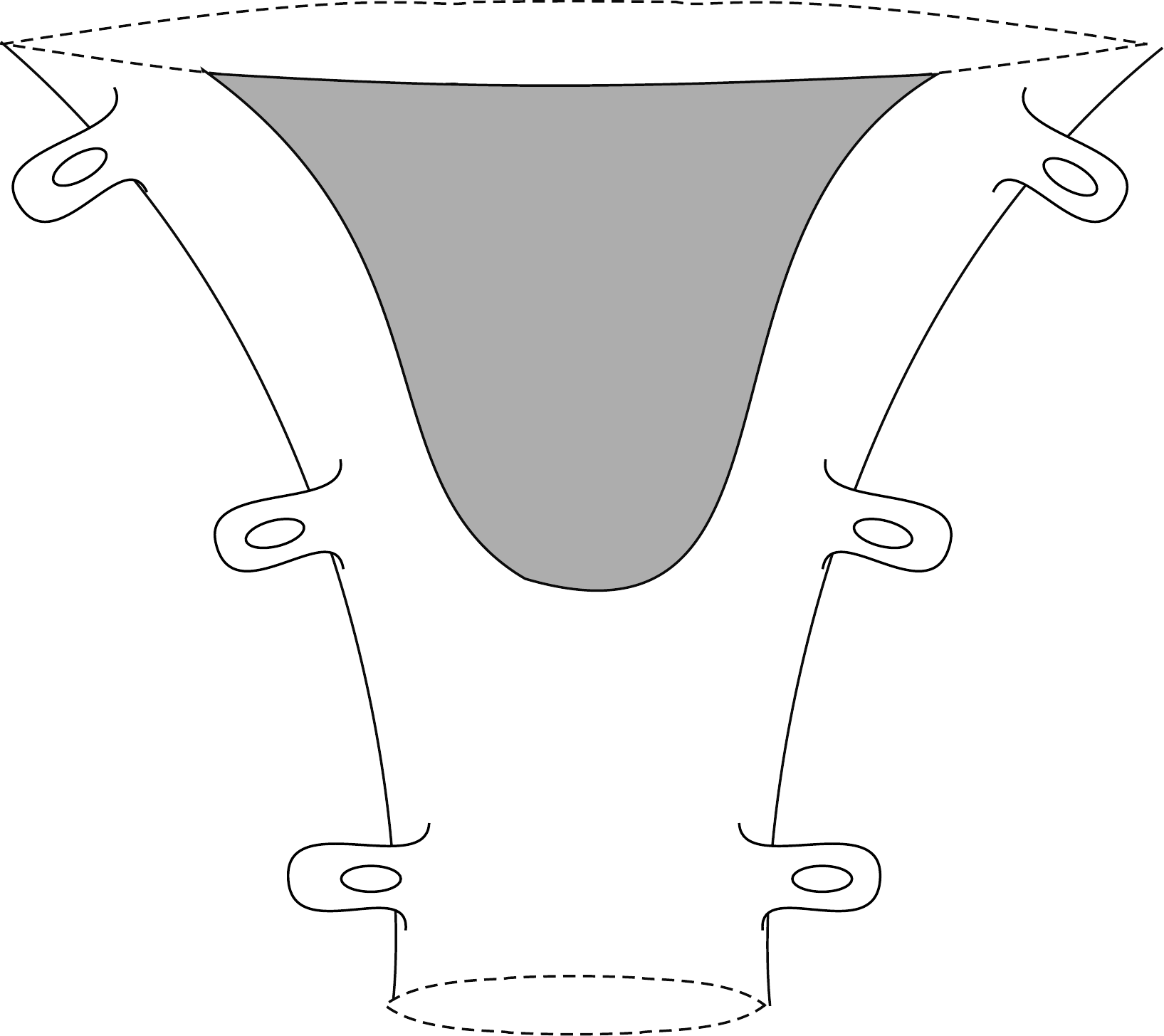}
\put(-85,115){\small$H$}

\end{picture}
\caption{A hyperbolic half plane $H$ replaces $H'$}\label{hypsemipl}
\end{center}
\end{figure}

\section{Limit Points and the Ideal Boundary}\label{idealb}
We will let $\Delta$ denote the open unit disk with the Poincar\'e metric. The closed unit disk will be denoted by $\mathbb{D}^{2}=\Delta\cup S^{1}_{\infty}$, where $ S^{1}_{\infty}$ is the unit circle, called the circle at infinity. If $L$ is a complete hyperbolic surface without boundary, the universal covering space $\widetilde L$ is $\Delta$. 
If $\partial L\ne\emptyset$ and is geodesic, the double $2L$ has a canonical hyperbolic metric which is the double of the metric on $L$ and $\widetilde{2L}=\Delta$.  If one fixes a lift $\widetilde L\subset\Delta$ of $L$, this serves as the universal covering space.  The group of deck transformations is the restriction to $\widetilde L$  of the subgroup  of deck transformations for $p:\Delta\to2L$ which leaves $\widetilde L$ invariant.

We will write $\widehat L$ for the closure of $\widetilde L$ in $\mathbb{D}^{2}$. We will set $E=\widehat L\cap S^{1}_{\infty}$ and call this the \emph{ideal boundary} of $\widetilde L$.  Thus $E$ is a compact subset of $ S^{1}_{\infty}$ and is equal to the unit circle precisely when $\partial L=\emptyset$.

 Many results are more easily proven for surfaces without boundary.  One often proves a result for $L$ by proving the corresponding result for $2L$ and remarking that its truth there implies its truth in $L$.  For this, the following lemma will be needed.

\begin{lemma}\label{2Lstand}
 If $\partial L\ne\emptyset$, then $L$ is  a standard hyperbolic surface if and only if $2L$ is a standard hyperbolic surface.
 \end{lemma}
 
 \begin{proof}
 The ``if'' direction is trivial.  For the converse, 
 suppose $L$ has a standard hyperbolic metric. Write $2L=L\cup L'$ and consider the tiling $\{L_{i}\}_{i=1}^{\infty}$ of $\Delta$ by lifts  $L$ and $L'$.  Let $A_{i}$ denote the set of ideal endpoints of $\partial L_{i}$, either a finite set for all $i\ge1$, or countably infinite for all $i\ge1$.  Since $\Delta=\bigcup_{i=1}^{\infty}L_{i}$ and $L$ (hence $L'$) has no half planes, it is elementary that $A_{*}=\bigcup_{i=1}^{\infty}A_{i}$ is a countable, dense subset of $S^{1}_{\infty}$.
 
Let $\delta>0$ be given. If $A_{i}$ is countably infinite, then all but finitely many of the components of $\partial L_{i}$ are sets with \emph{Euclidean} diameter $<\delta$.  (We use the standard Euclidean metric on the open unit disk.) Similarly, whether or not the $A_{i}$'s are infinite, the components of $\partial L_{i}$  have Euclidean diameter $<\delta$ for $i$ sufficiently large.

Suppose that the doubled metric on $2L$ is not standard. Thus, $2L$ contains a half plane $H$ with lift $\widetilde{H}$ having a nondegenerate arc $\alpha\subset S^{1}_{\infty}$ as its ideal boundary.  The projection of $\widetilde H$ onto $H$ by the covering map is a homeomorphism.  Let $x$ be an interior point of $\alpha$.  Since $A_{*}$ is countable and dense in $S^{1}_{\infty}$, we find a sequence of points $x_{n}\in A_{*}\cap\alpha$ converging to $x$.  If infinitely many of the $x_{n}$'s lie in a single $A_{i}$, then for some integer $n$, the component $\sigma$ of $\partial L_{i}$ having $x_{n}$ as an ideal endpoint lies entirely in $\widetilde H$.  Even if not, it remains true that there is an integer $n$ such that the component $\sigma$ of $\partial L_{i}$ having $x_{n}$ as an ideal endpoint lies entirely in $\widetilde H$.  Since the $L_{i}$'s tile $\Delta$, the half plane $H'\subset \widetilde H$ bounded by $\sigma$ contains infinitely many $L_{i}$'s which are lifts of $L$.  Consequently, there are covering transformations permuting these $L_{i}$'s nontrivially, hence identifying distinct points of $\widetilde H$ under the projection to $2L$.  This contradiction implies that $2L$ is standard.
\end{proof}

\begin{definition}\label{limst}

The \emph{limit points} of $L$ are the accumulation points in $S^1_{\infty}$ of the set  $\{\gamma(x_0) \mid \gamma\text{ a deck transformation of } \widetilde{L}\}$ for fixed $x_0\in \widetilde{L}$. The union $Y$ of these points is the limit set of $L$.

\end{definition}

The following is well known and elementary.

\begin{lemma}

The limit set of $L$ is independent of $x_0\in \widetilde{L}$.

\end{lemma}

Let $X\subset S^{1}_{\infty}$ be the set of fixed points of the (extensions to $\widehat L$ of the) nontrivial deck transformations of $\widetilde{L}$.  Then $X\subset Y\subseteq E$.  

\begin{definition}
We will call an end of $L$ a \emph{simply connected end} if it has a  simply connected neighborhood in $L$. 
\end{definition}

 Simply connected ends can occur only if $\partial L\ne\emptyset$.   For instance, an isolated simply connected end has a neighborhood homeomorphic to $[0,1]\times[0,\infty)$.  But the simply connected ends can form a very complicated subset of the endset, even a Cantor set of such ends being possible.

\begin{theorem}\label{threetwo}
If $L$ is a standard hyperbolic surface with no simply connected ends, then $X$ is dense in the ideal boundary $E$ and $Y=E$.
\end{theorem}

\begin{proof}
If $L$ is nonorientable, its orientation cover is intermediate between $L$ and $\widetilde L$, hence we can assume that $L$ is orientable. Now suppose the contrary, i.e.~that there is a point $e\in E\smallsetminus X$ not approached by points in $ X$. If $X=\emptyset$, our surface is simply connected and has been excluded by  hypothesis. Let $A$ be the maximal open interval in $S^1_{\infty}\smallsetminus X$ containing the point $e\in E$, letting $a,b\in\overline{X}$ denote the endpoints of $A$. We consider the cases $a=b$ and $a\ne b$.

If $a=b$, then all deck transformations would be parabolics fixing this point.  In this case, the covering group $G$ is infinite cyclic and $L=\widetilde L/G$.  If $\partial L=\emptyset$, this is the open annulus with one end a cusp and the  other a ``flaring'' annular end (an infinite hyperbolic trumpet).  This surface contains an imbedded half plane and has been excluded by  hypothesis. If $\partial L\ne\emptyset$, then $L$ has one cusp and, since the boundary is geodesic,  at least one simply connected end, also excluded by hypothesis.

Suppose then that $a\ne b$. Let $C$ be the geodesic with endpoints $a,b$.  Since $X$ lies in the ideal boundary of $\widetilde L$, so do $a$ and $b$, hence  $C\subset\widetilde{L}$. Let $H'$ be the portion of $\mathbb{D}^2$ bounded by $C\cup A$, and set $H=H'\cap\widetilde{L}$.  Since $A$ contains $e\in E$, $H\ne\emptyset$.  Remark that, if $g$ is a deck transformation, either $g(H)\cap H=\emptyset$, or $g(H)=H$ and g is  hyperbolic with axis $C$.  Otherwise, $A\cap X\ne\emptyset$. 

We consider  two cases: (a)  $C$ is nondegenerate and is not the axis of a deck transformation,  or (b) $C$ is  the axis of a deck transformation $g$.  

In case~(a),  the above remark assures us that no deck transformation identifies distinct points of $H$ or of $\intr C$. Therefore, under the covering projection, $C$ projects to s a geodesic $C'\subset L$ homeomorphic to the reals which cuts off the  image of $H$ under the covering projection, which is a homeomorphism on $H$. If $\partial L=\emptyset$, the image of $H$ is a half plane, contradicting the fact that $L$ contains no half planes. If $\partial L\ne\emptyset$, the image of $H$  contains at least one simply connected end of $L$, again a contradiction.. 

In case~(b), the geodesic $C$ projects to an essential closed curve $C'\subset L$ which cuts off the image of $H$ in $L$.     Thus, if $\partial L=\emptyset$,  $H$ projects to a neighborhood $H/g$ of a flaring annular end of $L$, contrary to hypothesis. If  $\partial L\ne\emptyset$,  $H$ projects to a neighborhood $H/g$  of at least one simply connected end of $L$, contrary to hypothesis.

Thus, $E=\overline{X}\subseteq Y\subseteq E$, so we have also proven that $Y=E$.
\end{proof}

As this proof reveals, simply connected ends cause essentially the same obstruction to this theorem as imbedded half planes.

For the sake of completeness we include the standard proof of the following.

\begin{corollary}\label{orbitxdense}
If $L$ is a standard hyperbolic surface with no simply connected ends and $e\in E$, then the orbit of $e$ under the group of covering transformations is dense in $E$.
\end{corollary}

\begin{proof}
We will show that the orbit clusters at every point of the dense set $X$. Let $x\in X$ and $g$ a covering transformation transformation fixing $x$. Choose $e'$ in the orbit of $e$ not fixed by $g$. Since $g^{2}$ is either hyperbolic or parabolic, applying the positive and negative iterates of $g^{2}$ to $e'$ produces a subset of the orbit of $e$ clustering at $x$.
\end{proof}

\begin{corollary}\label{Cset}
If $L$ is a standard hyperbolic surface with no simply connected ends and $\widetilde{L}\ne\Delta$, then the ideal boundary  $E$ is a Cantor set.
\end{corollary}

\begin{proof}
By Corollary~\ref{orbitxdense}, every point of $E$ is a limit point of $E$. If $E$ is not a Cantor set, it contains a closed,  nondegenerate interval. Let $A$, with endpoints $a\ne b\in E$, be a maximal such interval. Let   $x\in\intr A$ be fixed by a deck transformation   $\gamma$.  Clearly, $\gamma(A) = A$ so $A=E=S^1_{\infty}$ which is a contradiction.
\end{proof}

\begin{remark}
Theorem~\ref{threetwo} and its corollaries are standard for complete hyperbolic surfaces of finite area.
\end{remark}

\begin{remark}
Remark that, even if $L$ has simply connected ends, its double $2L$ does not and so the conclusion of Theorem~\ref{threetwo} holds for $2L$.  In applications of this theorem in what follows, we will always be working either in $2L$ or in $L$ if $\partial L=\emptyset$.
\end{remark}

\section{Extensions of Homeomorphisms to the Ideal Boundary}

\begin{theorem}\label{extcont}
If $L$ is a standard hyperbolic surface and  $h:L\rightarrow L$ is  a homeomorphism, then any lift $\tilde{h}:\widetilde{L}\to\widetilde{L}$ extends canonically to a homeomorphism $\widehat{h}:\widehat{L}\to\widehat{L}$.
\end{theorem}

\begin{remark}

If $L$ is not standard, that is contains imbedded half planes, then there are homeomorphisms $\widetilde{h}$ that do not extend.

\end{remark}

If  $\partial L\ne\emptyset$, then $h$ sends geodesics in $\partial L$ to geodesics in $\partial L$ and therefore, if    the assertion holds for the double $2L$, it holds for $L$.  Thus, we may assume that $\partial L=\emptyset$.  If $L$ is nonorientable, $h$ admits orientation reversing lifts $\widetilde h$ as well as orientation preserving ones.  If $L$ is orientable, then $h$ is orientation preserving (respectively, reversing)  if and only if  all of its lifts $\widetilde h$ are orientation preserving (respectively, reversing).    

\begin{lemma}\label{olh'}
Any lift $\widetilde h:\Delta\to\Delta$ induces a homeomorphism $\overline h:S^{1}_{\infty}\to S^{1}_{\infty}$ which is orientation preserving or reversing according as $\widetilde h$ is orientation preserving or reversing.
\end{lemma}

\begin{proof}
Recall that $X\subset S^{1}_{\infty}$ is the set of fixed points of nontrivial deck transformations and, since $\partial L=\emptyset$, $X$ is dense in $S^{1}_{\infty}$ (Theorem~\ref{threetwo}).  Since $X$ consists of endpoints of the lifts of  essential closed curves, it is clear that $\widetilde h$ induces a bijection  $\overline h:X\to X$.  This bijection preserves or reverses the cyclic order induced on $X$ by the orientation of $S^{1}_{\infty}$ according as $\widetilde h$ is orientation preserving or reversing.  In either case, it is well known that the fact that $X$  is dense in $S^{1}_{\infty}$ implies that $\overline h$ extends to a homeomorphism, again denoted by $\overline h:S^{1}_{\infty}\to S^{1}_{\infty}$ with the asserted orientation properties.
\end{proof}

It is now clear how to define $\widehat h:\widehat L\to\widehat L$.   We set $\widehat h|\Delta=\widetilde h$ and $\widehat h|S^{1}_{\infty}=\overline h$. Since $\widehat L=\mathbb D^{2}$ is compact, this bijection will be a homeomorphism precisely if it is continuous at each point $z\in S^{1}_{\infty}$.

\begin{definition}\label{psg}
If $\gamma\subset L$ is a curve such that some (hence every) lift $\widetilde{\gamma}$ has two  \textbf{distinct}, well defined ideal endpoints on $ S^{1}_{\infty}$, $\gamma$ is called a \emph{pseudo-geodesic}. We denote  the extension of the lift $\widetilde{\gamma}$ to $\mathbb{D}^{2}=\Delta\cup S^{1}_{\infty}$ by $\widehat{\gamma}$ and call it a completed lift of $\gamma$.
\end{definition}

Remark that a completed lift $\widehat\gamma$ is actually the metric completion of $\widetilde\gamma$ relative to the standard Euclidean metric on $\Delta$.

It is well known that a loop $\sigma:S^{1}\to L$ will be a pseudo-geodesic if and only if it is essential and does not bound a cusp.  Remark that the endpoints of completed lifts  $\widehat\sigma$ depend only on the free homotopy class of $\sigma$. 

We regularly write $h(\sigma)$ for $h\circ\sigma$, where $\sigma: S^{1}\to L$.

\begin{definition}\label{defn6}
$$\mathcal{Z} = \{z\in S^{1}_{\infty}\mid   \text {$z$ is an endpoint of $\widehat\sigma$ , $\sigma$ a closed pseudo-geodesic}\}.$$
\end{definition}

\begin{remark}

Because the lifts of a closed pseudo-geodesic have distinct endpoints, a closed pseudo-geodesic is homotopic to a  closed geodesic. In particular, closed pseudo-geodesics are not parabolic curves. 

\end{remark}

\begin{corollary}\label{olh}
The homeomorphism $\overline h:S^{1}_{\infty}\to S^{1}_{\infty}$ carries   $\mathcal{Z}$ bijectively onto $ \mathcal{Z}$.
\end{corollary}

Indeed, $\widetilde h$ carries the lifts of pseudo-geodesic loops exactly onto lifts of pseudo-geodesic loops, so the assertion is a consequence of (the proof of) Lemma~\ref{olh'}.

\begin{definition}
Let $\mathcal{G}$ be the set of pseudo-geodesics in $L$ such that some, hence every, completed lift $\widehat\gamma$ has both endpoints in $\mathcal{Z}$.
\end{definition}

\begin{lemma}\label{inGG}
if $\gamma\in\mathcal{G}$, then $h(\gamma)\in\mathcal{G}$.
\end{lemma}

\begin{proof}
Fix a lift $\widetilde h:\Delta\to\Delta$ and let $\overline{h}:\mathcal{Z}\to \mathcal{Z}$  be the homeomorphism given by Corollary~\ref{olh}.
If $\widehat\gamma$ has an endpoint $z\in \mathcal{Z}$, we will show that $\widetilde{h}(\widetilde{\gamma})$ limits on the point $\overline{h}(z)\in \mathcal{Z}$.  Applying this to both endpoints of $\widehat\gamma$, $\gamma\in\mathcal{G}$, will prove the lemma.

Let $\sigma$ be a closed geodesic with a lift $\widetilde\sigma$ such that one endpoint of $\widehat\sigma$ is $z$.  The existence of such a closed geodesic $\sigma$ follows from Definition~\ref{psg} and the fact that $z\in\mathcal{Z}$. Let $\tau$ be a  closed geodesic intersecting $\sigma$ transversely and let  $a$ be one of these intersection points.  Let $\{a_{n}\}_{n\in\mathbb{Z}}$ be the lifts of $a$ in $\widetilde\sigma$, indexed so that $\lim_{n\to\infty}a_{n}=z$ monotonically along $\widehat\sigma$. Let $\tau_{n}$ be the lift of $\tau$ through $a_{n}$.  Each $\widehat{\tau}_{n}$ has endpoints $u_{n},w_{n}$ bounding a subarc $\alpha_{n}\subset S^{1}_{\infty}$ containing the point $z$.  The sequences $u_{n}\to u$
 and $w_{n}\to w$ monotonically as $n\to\infty$.  We claim that $u=z=w$.  Otherwise, the geodesics $\tau_{n}$ accumulate locally uniformly  (in the hyperbolic metric) on the geodesic $\tau'$ with endpoints $u,w$.  This would mean that the closed geodesic $\tau$ in $L$ accumulates locally uniformly on a distinct geodesic. This is impossible. The sets $V_{n}\subset\mathbb{D}^{2}$ bounded by $\tau_{n}\cup\alpha_{n}$ form a fundamental system of closed neighborhoods of $z$.  
 
Applying $\widetilde h$ to this picture gives a family $\widetilde{h}(\tau_{n})$ of pseudo-geodesics with endpoints $\overline{h}(u_{n}),\overline{h}(v_{n})\in S^{1}_{\infty}$, meeting the curve $\widetilde h(\widetilde\sigma)$ in points $\widetilde{h}(a_{n})$, such that $\overline{h}(u_{n})\to\overline{h}(z)$ and $\overline{h}(v_{n})\to\overline{h}(z)$ monotonically (by the fact that $\overline h$ is a homeomorphism that either preserves or reverses orientation). Let $U_{n}$ be the subset of $\mathbb{D}^{2}$ bounded by the arc $\widetilde{h}(\tau_{n})$ and the arc of $S^{1}_{\infty}$ with endpoints $\overline{h}(u_{n}),\overline{h}(v_{n})$ containing the point $\overline{h}(z)$. 
Since $\widetilde{h}:\Delta\to\Delta$ is a homeomorphism,  $\widetilde{h}(V_{n}\cap\Delta) = U_{n}\cap\Delta$ is a nest with empty intersection. Also $\overline{h}(u_{n})\to\overline{h}(z)$ and $\overline{h}(v_{n})\to\overline{h}(z)$ and it follows that the $U_{n}$ form a descending nest of closed neighborhoods    of $\overline{h}(z)$. If $U\subset\mathbb{D}^{2}$ is an open neighborhood of $\overline{h}(z)$ and no $U_{n}\subset U$, then, by compactness of $\mathbb{D}^{2}\setminus U$, there exists a point $y\in U_{n}\setminus U$ for all $n\in\mathbb{Z}$. Since $\overline{h}(u_{n})\to\overline{h}(z)$ and $\overline{h}(v_{n})\to\overline{h}(z)$, it follows that $y\notin S^{1}_{\infty}$. Thus, $y\in\Delta$. Since $\widetilde{h}:\Delta\to\Delta$ is a homeomorphism, it follows that there exists an $x\in\Delta$ such that $\widetilde{h}(x) = y$. Since the $V_{n}$ form a fundamental system of neighborhoods of $z$, it follows that $x\notin V_{n_{0}}$ for some $n_{0}$. Thus, $y\notin U_{n_{0}}$ which is a contradiction so the $U_{n}$ form a fundamental system of neighborhoods of $\overline{h}(z)$. 
 Since $\widehat\gamma$ has endpoint $z$, a neighborhood $N$ of one end of $\widetilde\gamma$ lies in $V_{n}\cap\Delta$. Therefore $\widetilde{h}(N)\subset U_{n}\cap\Delta$.   It follows that $\widetilde{h}(\widetilde{\gamma})$ has endpoint $\overline{h}(z)\in \mathcal{Z}$.
\end{proof}

\begin{remark}

Crucial in the previous proof is the existence of a closed geodesic $\tau$ intersecting the closed geodesic $\sigma$ transversely. The existence of such a $\tau$ can be seen as follows. If $\sigma$ self intersects, let $\tau=\sigma$. Otherwise $\sigma$ is a simple closed geodesic. If it does not intersect another closed geodesic, then $\sigma$ is peripheral. Since $\partial L = \emptyset$ and $L$ is standard, $\sigma$ would have to bound a cusp on one side. But in this case $\sigma$ is not a geodesic, contrary to hypothesis.

\end{remark}

Since $\mathcal{Z}$ is a union of orbits in $ S^{1}_{\infty}$ of the group of (extended) covering transformations and since $\partial L=\emptyset$, we can apply Corollary~\ref{orbitxdense} to obtain the following.

\begin{lemma}
The set $\mathcal{Z}$ is dense in $ S^{1}_{\infty}$.
\end{lemma}

\begin{proof}[Proof of \emph{Theorem~\ref{extcont}}]
We want to show that $\widehat{h}:\mathbb{D}^{2}\to\mathbb{D}^{2}$ is continuous at $z\in\mathbb{D}^{2}$.  This is clear if $z\in \Delta$, so we assume $z\in S^{1}_{\infty}$. Let $U$ be an open neighborhood of $\widehat{h}(z)$. Since $\mathcal{Z}$ is dense in $S^{1}_{\infty}$, we can choose  a compact arc $[a.b]\subset U\cap S^{1}_{\infty}$ such that $\widehat{h}(z)\in(a.b)$.
Let $\widehat{\gamma}$ be any curve in $U$ with endpoints $a$ and $b$ and $\gamma$ the projection of $\widehat{\gamma}\cap\Delta$ to $L$. Then $\gamma\in\mathcal{G}$ so by Lemma~\ref{inGG}, $h^{-1}(\gamma)\in\mathcal{G}$. 

Therefore the curve,    $\widehat{h}^{-1}(\widehat{\gamma})$, which is an extended lift of  $h^{-1}(\gamma)$, has  endpoints $\overline{h}^{-1}(a),\overline{h}^{-1}(b)$.  Then the subset $V$ of $\mathbb{D}^{2}$ bounded by $\widehat{h}^{-1}(\widehat{\gamma})$ and $\widehat{h}^{-1}(\alpha)$  is a closed neighborhood of $z$ in $\mathbb{D}^{2}$ and $\widehat{h}(V)\subset U$.  This proves continuity at arbitrary $z\in S^{1}_{\infty}$.
\end{proof}

\begin{corollary}
If $\gamma$ is a pseudo-geodesic in $L$ and $h:L\to L$ is a homeomorphism, then $h(\gamma)$ is a pseudo-geodesic.
\end{corollary}

\begin{remark}
Our proof of Theorem~\ref{extcont} includes the case in which $L$ is compact, hence gives a  fundamentally different  proof in that case from the ones given by Casson and Bleiler~\cite[Lemma~3.7]{bca} and Handel and Thurston~\cite[Corollary~1.2]{ha:th}.  These proofs make  use of compactness, whereas we do not.   The analogous result holds for higher dimensional, compact, hyperbolic manifolds~\cite[Proposition~C.1.2]{ben-pet}, \cite[Theorem~11.6.2]{rat}, where compactness is only used to guarantee that $\widetilde h$ is a pseudo-isometry \cite[p. 555]{rat}.  In~\cite{bca}, compactness is only used to guarantee that $\widetilde h$ is uniformly continuous.
\end{remark}

\section{The Basic Isotopy Theorem.}\label{isotheorem}

The following  is   well known when $L$ is compact.

\begin{theorem}\label{isoisota}
If $L$ is a standard hyperbolic surface and  $h:L\to L$ is a homeomorphism, then $h$ is isotopic to the identity if and only if it has a lift to $\widetilde L $ such that $\widehat{h}|E$ is the identity. 
\end{theorem}

The ``only if'' part of this theorem is elementary. In fact, working in the double $2L$, if $h$ is isotopic to the identity, it is obvious that $\overline{h}$ is the identity on $\mathcal{Z}$ and since $\mathcal{Z}$ is dense in $S^{1}_{\infty}$ then  $\overline{h} = \id$ and this remains true for $\overline h|E$.

\begin{corollary}\label{isoisot}
If $L$ is a standard hyperbolic surface  and  $f,g:L\to L$ are homeomorphisms, then $f$ is isotopic to $g$ if and only if there are lifts to $\widetilde L$ such that  $\widehat{f},\widehat{g}$ agree on the ideal boundary $E$. 
\end{corollary}

Indeed, set $h=g^{-1}\circ f$.

 We cannot find a proof of Theorem~\ref{isoisota} in the literature for surfaces of infinite Euler characteristic, so we give a detailed sketch here.  We do not need orientability nor empty boundary.  Basic to our proof are the following two Epstein-Baer theorems.  They are, respectively, Theorem~2.1 and Theorem~3.1 in~\cite{Epstein:isotopy}.  While the proofs are carried out in the PL category, it is shown in the Appendix of~\cite{Epstein:isotopy} that these and other results in the paper remain true in the topological category. There is no hyperbolic metric assumed on $L$ in these two theorems.
 
\begin{theorem}[Epstein-Baer]\label{2.1}
Let $\alpha,\beta: S^{1}\to\intr L$ be freely homotopic, imbedded, $2$-sided, essential circles.  Then there is a compactly supported homeomorphism $\varphi:L\to L$ and an ambient isotopy $\Phi:L\times I\to L$ with $\Phi(\cdot,0) =\id_{L}$ and $\Phi(\cdot,1)=\varphi$, compactly supported in $(\intr L)\times I$, such that $\varphi\circ\beta=\alpha$.
\end{theorem}

\begin{theorem}[Epstein]\label{3.1}
Let $\alpha,\beta:[0,1]\to L$ be properly imbedded arcs with the same endpoints which are homotopic modulo the endpoints.  Then there is a compactly supported homeomorphism $\varphi:L\to L$ and an ambient isotopy $\Phi:L\times I\to L$ with $\Phi(\cdot,0) =\id_{L}$ and $\Phi(\cdot,1)=\varphi$, compactly supported in $L\times I$ and carrying $\partial L\times I\to\partial L$ by $\Phi(x,t)=x$, such that $\varphi\circ\beta=\alpha$.
\end{theorem}

We will generally call $\varphi$ itself a compactly supported ambient isotopy.   As noted earlier, we sometimes abuse terminology by identifying a curve $\alpha$ with its image.  Thus, instead of writing $g\circ\alpha=\alpha$, we might write $g|\alpha=\id_{\alpha}$ (or $=\id$).

\subsection{Preliminaries}

The following is elementary and well known.

\begin{lemma}\label{isotSR}
Orientation preserving homeomorphisms $h: S^{1}\to S^{1}$ and $h:\mathbb{R}\to\mathbb{R}$ are isotopic to the identity.
\end{lemma}

\begin{corollary}\label{isotid}
The homeomorphism $h:L\to L$ of \emph{Theorem~\ref{isoisota}} admits an ambient isotopy $\varphi$, supported near $\partial L$, such that $\varphi\circ h$ restricts to the identity on $\partial L$.
\end{corollary}

 \begin{proof}
 Since $\widehat h$ fixes $E$ pointwise, it is clear that $h$ preserves the components of $\partial L$ and is orientation preserving on each.  In a collar neighborhood of each boundary component, one extends the isotopy of Lemma~\ref{isotSR} to an ambient isotopy supported in the collar.
 \end{proof}
 
 From now on, therefore, we will assume that $h|\partial L=\id_{\partial L}$.
 
 \begin{lemma}\label{uniqha}
 If $h$ admits a lift $\widetilde h$ such that $\widehat{h}|E=\id_{E}$, then it admits a unique such lift.
 \end{lemma}
 
 \begin{proof}
 Let $g,f:\widetilde{L}\to\widetilde L$ be lifts of $h$ so that $\widehat{f}|E=\widehat{g}|E=\id_{E}$.  Then $f\circ g^{-1}$ is a covering transformation $\psi$ which is the identity on $E$, hence $\psi=\id$.  
 \end{proof}
 
  \begin{remark}\label{remfive}
  Hereafter, $h^{*}$ will denote this unique lift. By abuse, its extension to $\widehat L=\widetilde{L}\cup E$ will also be denoted by $h^{*}$. Remark that, if $h$ is varied by an ambient isotopy $\varphi$, then 
 $\varphi\circ h$ also admits a lift $(\varphi\circ h)^{*}$.  Indeed, $\varphi$ itself has such a lift, since it preserves the free homotopy classes of loops, and so   $(\varphi\circ h)^{*}=\varphi^{*}\circ h^{*}$. 
 \end{remark}

 \begin{corollary}\label{commcov}
 The canonical lift $h^{*}$ commutes with all covering transformations.
 \end{corollary}
 
 \begin{proof}
 Indeed, if $\psi$ is a covering transformation, $\psi\circ h^{*}\circ\psi^{-1}$ is a lift of $h$ which fixes $E$ pointwise.  By Lemma~\ref{uniqha}, $\psi\circ h^{*}\circ\psi^{-1}=h^{*}$.
 \end{proof}

 The following is a key lemma.
 
 \begin{lemma}\label{key}
 Let $\sigma: S^{1}\to L$ be an essential imbedded circle such that $h|\sigma=\id_{\sigma}$.  Then there is an ambient isotopy $\varphi$, supported in an arbitrarily small neighborhood of $\sigma$, such that  $(\varphi\circ h)^{*}$  is the identity on any lift $\widetilde\sigma$ of $\sigma$.
 \end{lemma}
 
 \begin{proof}
 Since $h(\sigma)=\sigma$ and since $ h^{*}$ preserves the endpoints of $\widetilde\sigma$ in $E$, $h^{*}(\widetilde{\sigma})=\widetilde{\sigma}$.  If we set $ S^{1}=\mathbb{R}/\mathbb{Z}$, this parametrizes $\sigma$ and so parametrizes $\widetilde\sigma$ as $\mathbb{R}$ (up to an additive integer).  The fact that $h|\sigma=\id$ then implies that $h^{*}|\widetilde{\sigma}$ is just translation of $\mathbb{R}$ by an integer $n$. Fix a normal neighborhood $N=\mathbb{R}\times[-1,1]$ of $\widetilde\sigma$, with $\widetilde{\sigma}=\mathbb{R}\times\{0\}$. Do this so that $N$ is a lift of an annular neighborhood $A$ of $\sigma$.  If $\sigma$ is a component of $\partial L$, make $N$ a 1-sided normal neighborhood $\mathbb{R}\times[0,1]$.  Now translation on $\mathbb{R}$ by $-n$ is isotopic to the identity through translations and this easily defines an ambient isotopy $\varphi$, supported in $N$, such that ${\varphi}^{*}\circ h^{*}$ is the identity along $\widetilde\sigma$.  As the notation suggests, we have taken care that $\varphi^{*}$ is a lift of an ambient isotopy $\varphi$ supported in $A$ which leaves $\sigma$ (as a point set) invariant.  If $\widetilde{\sigma}_{1}$ is another lift of $\sigma$, it is the image of $\widetilde\sigma$ under a deck transformation $\psi$. The parametrization of $\widetilde{\sigma}_{1}$ corresponds to that of $\widetilde\sigma$ under $\psi$ up to translation by an integer.   By Corollary~\ref{commcov} and Remark~\ref{remfive}, $(\varphi\circ h)^{*}$ commutes with $\psi$. Then for any $y\in\widetilde{\sigma}$,
 $$\psi(y) =\psi\circ(\varphi\circ h)^{*}(y) = (\varphi\circ h)^{*}(\psi(y)).$$
Thus, $(\varphi\circ h)^{*}$ fixes $\widetilde{\sigma}_{1}$ pointwise. 
\end{proof}
  
 \begin{corollary}
 There is an ambient isotopy $\varphi$, supported in any preassigned neighborhood of $\partial L$, such that $(\varphi\circ h)^{*}|\partial\widetilde L=\id_{\partial\widetilde L}$.
 \end{corollary}
 
 \begin{proof}
 For each compact component of $\partial L$, apply Lemma~\ref{key}.  The lifts of noncompact components $\ell$ project one-to-one onto $\ell$ and the assertion is trivial.
 \end{proof}
 
   From now on, therefore, we will assume both that $h|\partial L=\id_{\partial L}$ and $h^{*}|\partial\widetilde{L}=\id_{\partial\widetilde L}$.

 \subsection{The proof of Theorem~\ref{isoisota}}
We will concentrate on the case in which $L$ is noncompact, remarking at the end on how our methods adapt easily to the compact case.

 There is an  exhaustion $K_{0}\subset K_{1}\subset\cdots\subset K_{i}\subset\cdots$ of $L$
by compact, connected surfaces $K_{i}$, with $K_{i}\subset\intr K_{i+1}$, $i\ge0$.  We can require that the finitely many connected components  of $L\smallsetminus K_{i}$ are neighborhoods of ends (that is, are non-compact with compact frontier), the frontier of each being a finite, disjoint union of properly imbedded  arcs and 2-sided, essential  circles.   Let the collection of all of these arcs and circles be denoted by $\mathfrak{A}=\{\alpha_{i}\}_{i=1}^{\infty}$.  Note that $\mathfrak{A}$ partitions $L$ into a family of compact submanifolds with boundary composed of elements of $\mathfrak{A}$ and arcs and/or circles in $\partial L$.  Enumerate these submanifolds as $\{B_{k}\}_{k\ge0}$, in such a way that $B_{0}=K_{0}$, $B_{1},B_{2},\dots,B_{k_{1}}$ are the components of $K_{1}\smallsetminus\intr K_{0}$, $B_{k_{1}+1},B_{k_{1}+2},\dots,B_{k_{2}}$ are the components of $K_{2}\smallsetminus\intr K_{1}$, etc.

\begin{lemma}
There is an ambient isotopy $\psi:L\to L$ such that $\psi\circ h|\alpha_{i}=\id$ and $(\psi\circ h)^{*}|\widetilde{\alpha}_{i} = \id$, for each $\alpha_{i}\in\mathfrak{A}$ and each lift $\widetilde{\alpha}_{i}$. 
\end{lemma}

\begin{proof}
If $\alpha=\alpha_{i}\in\mathfrak{A}$   is a properly imbedded arc in $L$, then $h(\alpha)$ is properly imbedded with the same endpoints as $\alpha$, since $h|\partial L=\id_{\partial L}$.  Since $h^{*}|\partial\widetilde{L}=\id_{\partial\widetilde L}$, the lifts $\widetilde\alpha$ and $h^{*}(\widetilde\alpha)$ have the same endpoints.  It follows that $\alpha$ and $h(\alpha)$ are homotopic modulo their endpoints.  By Theorem~\ref{3.1}, there is a compactly supported ambient isotopy $\varphi$, keeping $\partial L$ pointwise fixed, such that $\varphi\circ h|\alpha=\id$.  Note that $\varphi$ perturbs only finitely many $h$-images of elements of $\mathfrak{A}$.  Note that each lift $\widetilde\alpha$ has endpoints in $\partial\widetilde L$, fixed by $(\varphi\circ h)^{*}$.  It is evident, then, that $(\varphi\circ h)^{*}|\widetilde\alpha =\id$.

Similarly,   if $\alpha\in\mathfrak{A}$ is a circle imbedded in $\intr L$, any lift $\widetilde\alpha$ is an imbedded copy of $\mathbb{R}$ in $\widetilde L$ having well defined ideal endpoints in $E$. Since $h^{*}$ fixes these endpoints, the lift $h^{*}(\widetilde\alpha)$ of $h(\alpha)$ has these same ideal endpoints and $\alpha$ is freely homotopic to $h(\alpha)$ in $\intr L$. By Theorem~\ref{2.1}, there is an  ambient isotopy $\varphi$, compactly supported in $\intr L$, such that $\varphi\circ h|\alpha=\id$.  Again, $\varphi$ perturbs only finitely many $h$-images of elements of $\mathfrak{A}$. An application of Lemma~\ref{key} shows that the ambient isotopy $\varphi$ can be modified so that $(\varphi\circ h)^{*}|\widetilde\alpha =\id$.

Let $\alpha_{1}\in\mathfrak{A}$ be contained in $\partial B_{0}$.  Construct $\varphi_{1}$ as above.  If $\alpha_{2}\subset\partial B_{0}$, we must choose $\varphi_{2}$ so $\varphi_{2}\circ\alpha_{1}=\alpha_{1}$.  The trick is to temporarily cut $L$ apart along (the image of) $\alpha_{1}$ and apply the above argument in the resulting component $L'$ containing $\alpha_{2}$. Suppose that $\alpha_{1},\alpha_{2},\dots,\alpha_{n}$ is the full list of elements of $\mathfrak{A}$ in $\partial B_{0}$.  Then continuing in this way, we find $\psi_{0}=\varphi_{n}\circ\varphi_{n-1}\circ\cdots\circ\varphi_{1}$, compactly supported, the identity on $\partial L$,  such that $\psi_{0}\circ h$ restricts to the inclusion map on $\partial B_{0}$ and $(\psi_{0}\circ h)^{*}$ restricts to the inclusion  map on $\partial\widetilde{B}_{0}$, for each lift $\widetilde{B}_{0}\subset\widetilde L$.  Now let $L'$ be the component of $L\smallsetminus\intr B_{0}$ containing $B_{1}$.  Work in $L'$ to produce $\psi_{1}$, a compactly supported isotopy on $L'$, the identity on $\partial L'$, such that $\psi_{1}\circ h$ restricts to the inclusion map on $\partial B_{1}$ and $(\psi_{1}\circ h)^{*}$ restricts to the inclusion on $\partial\widetilde{B}_{1}$, for each lift $\widetilde{B}_{1}\subset\widetilde L$.  It is important to note that $\psi_{1}$ can be viewed as a compactly supported isotopy on $L$ which is the identity on $B_{0}$ throughout the isotopy.  Proceeding in this way, produce an isotopy $\psi=\cdots\circ\psi_{j}\circ\psi_{j-1}\circ\cdots\circ\psi_{0}$ such that $\psi\circ h|\alpha$  and $(\psi\circ h)^{*}|\widetilde\alpha$ are the inclusions, for each $\alpha\in\mathfrak{A}$ and each lift $\widetilde\alpha$.  The isotopy is well defined since, for each $\ell\ge0$, all but finitely many $\psi_{j}$ are the identity isotopy on $B_{0}\cup B_{1}\cup\cdots\cup B_{\ell}$.
\end{proof}

Hereafter, we replace $h$ by $\varphi\circ h$, assuming that  $h$ is the identity  on $\partial L\cup\bigcup_{\alpha\in\mathfrak{A}}\alpha$ and that $h^{*}$ is the identity on the entire lift of this set.

\begin{proof}[Proof of \emph{Theorem~\ref{isoisota}}]
Each $B_{k}$ is a compact surface with nonempty boundary.  As such it can be viewed as the result of attaching finitely many bands (twisted and/or untwisted) to a disk (see~\cite[pp.~43-45]{mass}). Thus, there are finitely many disjoint,  properly imbedded arcs $\tau_{1},\dots,\tau_{r}$ in $B_{k}$  which decompose this surface into a disk.  The properly imbedded arcs $\tau_{i}$ and  $h\circ\tau_{i}$ have the same endpoints since $h$ fixes $\partial B_{k}$ pointwise. Furthermore, since $h^{*}$ fixes $\partial\widetilde{B}_{k}$ pointwise, they have lifts in  $\widetilde{B}_{k}$ with common endpoints.  This implies that   they are homotopic in $B_{k}$ by a homotopy keeping their endpoints fixed. Applying Theorem~\ref{3.1} in the usual way, allows us to assume that $h$ fixes $\tau_{i}$ pointwise, $1\le i\le r$. Since cutting $B_{k}$ apart along these arcs gives a disk $D$  and the homeomorphism $h':D\to D$  induced by $h$ is the identity on $\partial D$, we  apply Alexander's trick to find an isotopy of $h'$ to the identity which is constant on $\partial B$.  Regluing gives an isotopy of $h:B_{k}\to B_{k}$ to the identity which is constant on the boundary.   Carrying this out for each $B_{k}$, we obtain isotopies that fit together to an isotopy on $L$ since they leave all boundary components pointwise fixed.
\end{proof}

\begin{remark}
While the isotopy in the above proof is constant on $\partial L$, this is only after altering the original $h$ by an ambient isotopy that is not generally constant on $\partial L$.  Thus the isotopy in Theorem~\ref{isoisota} is not generally constant on the boundary.  A simple example is a Dehn twist on the closed annulus.
\end{remark}

\begin{remark}
If $L$ is compact, the above proof easily adapts.  Indeed, if $\partial L\ne\emptyset$, use our methods to isotope $h$ to be the identity on $\partial L$ and so that $h^{*}$ is the identity on $\partial\widetilde L$.  Now introduce the arcs $\tau_{i}$ as above and use Alexander's trick to complete the isotopy.  If $\partial L=\emptyset$, find a simple, closed, 2-sided curve $\sigma\subset L$ that does not disconnect ($\chi(L)<0$).  First do the isotopy of $h$ that makes it the identity on $\sigma$ and makes $h^{*}$ the identity on all lifts of $\sigma$.  Now cut apart along $\sigma$ and apply Alexander's trick. Since this last isotopy is constant on the boundary, we reglue to obtain the desired isotopy on $L$.
\end{remark}

\section{Homotopic Homeomorphisms}\label{homhom}

In~\cite{Epstein:isotopy}, Epstein states and proves the following theorem.  Here we say that a homotopy $H$ is rel $\partial L$ if the homotopy is through maps $H_{t}$ that carry $\partial L$ into itself.

\begin{theorem}[Epstein-Baer]\label{epsteinorig}
Let $L$ be a surface and  $h:L\to L$  a homeomorphism homotopic to the identity.  If $L$ is the open disk, the closed disk, the open annulus, or the closed annulus require that $h$ be orientation preserving. Suppose that either of the following two conditions is satisfied:
\begin{enumerate}
\item[$(1)$] The surface $L$ has all boundary components compact;
\item[$(2)$] The homotopy is rel $\partial L$ and proper.
\end{enumerate} Then $h$ is isotopic to the identity.  
\end{theorem}

\begin{corollary}\label{epsteincor}
Let $L$ be a surface and $f,g:L\to L$ homotopic homeomorphisms.  If $L$ is the open disk, the closed disk, the open annulus, or the closed annulus require either that $f$ and $g$ both preserve orientation or both reverse it.  If either condition $(1)$ or $(2)$ of \emph{Theorem~\ref{epsteinorig}} is satisfied, then $f$ is isotopic to $g$.  
\end{corollary}

\begin{proof}
If $H$ is the homotopy, then $g^{-1}\circ H$ is a homotopy of $g^{-1}\circ f$ to $\id_{L}$ and the hypotheses of Theorem~\ref{epsteinorig} are satisfied.  Therefore $g^{-1}\circ f$ is isotopic to the identity. 
\end{proof}

Using the hyperbolic methods of this paper, we will prove these results for standard surfaces, but with substantially weaker hypotheses.

A weaker condition than (1) is that $L$ has no simply connected ends.   It is simply connected ends that cause difficulties rather than noncompact boundary components. 
If $h$ is isotopic to the identity, then $h$ must fix all ends but there are many examples of standard surfaces (simply connected or not) admitting homeomorphisms that are homotopic to the identity but do not fix some simply connected ends. Here are the simplest ones. 

\begin{example}
Let $L$ be an ideal $n$-gon, a surface with $n\ge 3$ simply connected ends.   Since $L$ is contractible,   all maps $f:L\to L$ are in the same  homotopy class.  There is a homeomorphism $h:L\to L$ which permutes the ends cyclically.  Then $h$ is homotopic to the identity, but clearly not isotopic to the identity.  
\end{example}

Moreover, there are examples of homeomorphisms $h$ that are homotopic to the identity and fix all ends but are not isotopic to the identity. Part~(3) of Theorem~\ref{cc} shows that the following are the only possible surfaces that admit such examples.

\begin{example}\label{exfour}
If the surface $L$ is the connected sum along a loop $\gamma$ of any surface (other than the 2-sphere or closed disk or open disk) and a noncompact simply connected surface (other than the open disk) and if $h:L\to L$ is the homeomorphism given by  a non-trivial Dehn twist in an annular neighborhood of $\gamma$, then $h$ is homotopic to the identity (deform the simply connected summand into its interior, untwist, and deform back) but not isotopic to the identity. Note that $h$ fixes each simply connected end.
\end{example}

Recall that $X\subset E$ is the set of fixed points of nontrivial deck transformations.

\begin{lemma}\label{fixedX}
Let $L$ be a standard surface and $h:L\to L$ a homeomorphism that is homotopic to the identity.  Then $h$ admits a lift $\widetilde h:\widetilde L\to\widetilde L$ such that $\widehat h$ fixes $\overline X$ pointwise.
\end{lemma}

\begin{proof}
Let $H_{t}:L\to L$, $0\le t\le1$, be a homotopy such that $H_{0}=\id_{L}$ and $H_{1}=h$.  By the homotopy lifting property, there is a homotopy $\widetilde H_{t}:\widetilde L\to\widetilde L$ such that $\widetilde H_{0}=\id_{\widetilde L}$ and $\widetilde H_{1}=\widetilde h$ is a lift of $h$.  Let $\sigma$ be an essential closed loop in $L$, $\widetilde\sigma$ any lift.  If $\sigma$ is a pseudo-geodesic, the ideal endpoints of $\widetilde\sigma$ are distinct.  Otherwise they coincide.  In any event, these endpoints make up $X$ as $\sigma $ and its lifts vary.  Set $\widetilde\sigma_{t}=\widetilde H_{t}(\widetilde \sigma)$, a homotopy of $\widetilde \sigma$ to $\widetilde h(\widetilde\sigma)$ through lifts of $H_{t}(\sigma)$, $0\le t\le1$.  Since $X$ has empty interior, the ideal endpoints of $\widetilde\sigma_{t}$ are constant in $t$.  Consequently, $\widehat h$ fixes these endpoints, hence fixes $\overline X$ pointwise.
\end{proof}

\begin{theorem}\label{cc}
Let $L$ be a standard surface   and  $h:L\to L$  a homeomorphism which  is homotopic to the identity.  Suppose that any one of the following conditions is satisfied:
\begin{enumerate}
\item[$(1)$] $L$ has no simply connected ends;
\item[$(2)$] the homotopy  is rel $\partial L$;
\item[$(3)$]  $L$ is any   surface except those in Example~\ref{exfour} and the homeomorphism  $h$  fixes each simply connected end.
\end{enumerate}
Then $h$ is isotopic to the identity.
\end{theorem}

\begin{proof}  Assume condition~(1).  By Lemma~\ref{fixedX} and Theorem~\ref{threetwo}, there is a lift $\widetilde h$ such that $\widehat h$ fixes $E$ pointwise.  By Theorem~\ref{isoisota}, $h$ is isotopic to the identity.

Assume  condition~(2). The double $2L$ has no simply connected ends and both $h$ and the homotopy $H$ double because the homotopy is rel $\partial L$.  Apply the argument for condition~(1) to conclude that there is a lift $\widetilde{2h}$ such that $\widehat{2h}$ fixes $S^{1}_{\infty}$ pointwise.  Therefore $\widetilde h=\widetilde{2h}|\widetilde L$ leaves the universal cover $\widetilde L\subset\Delta$ invariant and $\widehat h$ fixes $E$ pointwise. Again, Theorem~\ref{isoisota} implies that $h$ is isotopic to the identity.

Assume condition~(3). If $X=\emptyset$, then $L$ is simply connected, the projection map from $\widetilde L$ to $L$ is one-one, and the unique lift $\widetilde h$ of $h$ fixes every end of $\widetilde L$. Thus $\widehat h$ fixes $E$ and Theorem~\ref{isoisota} implies that $h$ is isotopic to the identity. Therefore, suppose $X\ne\emptyset$ and let $\widetilde h$ be the lift given by Lemma~\ref{fixedX}.  Assume there is a point $x\in E\smallsetminus\overline X$ and let $(a,b)\subset S^{1}_{\infty}\smallsetminus\overline X$ be the component containing $x$.  Then $a,b\in\overline X$ are fixed by $\widehat h$.  

If $a=b$ then $a$ is the only point of $X$ and is the fixed point of a parabolic deck transformation $\varphi$. Then for a lift $\widetilde h$ of $h$ and a suitable choice of $n$, $\varphi^{n}\circ \widetilde h$ is a lift of $h$ which fixes  every lift of each simply connected end while still fixing $a$. Again, Theorem~\ref{isoisota} implies that $h$ is isotopic to the identity.

Next suppose $a\ne b$. Let $\gamma\subset\widetilde L$ be the geodesic with endpoints $a$ and $b$.  Let $U\subset\widetilde L$ be the closed  set cut off by $\gamma$ on the side determined by $(a,b)$.  Since $X\cap(a,b)=\emptyset$, the only way a covering transformation can identify distinct points in $U$ is for that transformation to have $\gamma$ as its axis. 
Let $\psi$ be the covering transformation, the powers of which give every covering transformation with axis $\gamma$.  If $\psi $ is orientation preserving, then $U'=U/\psi$ is of the form $L_{0}\smallsetminus D$, where $L_{0}$ is noncompact and simply connected and $D$ is an open disk.  The boundary circle $s=\gamma/\psi$ is a closed geodesic. If $s\subset\partial L$, then $X=\{a,b\}$ and as in the case where $X$ has only one   point, one can choose a covering $\widetilde h$ of $h$  such that $\widetilde h$ fixes all the lifts of the ends in $U$ while still fixing $a$ and $b$.  Again, Theorem~\ref{isoisota} implies that $h$ is isotopic to the identity. Otherwise,  $s$ separates  $L$  and   $L=L_{0}\# L_{1}$ is a connected sum violating the condition on $L$ in~(3).  If $\psi$ is orientation reversing, the reader will easily check that $L$ is the connected sum of a projective plane and a noncompact, simply connected surface, violating the condition on $L$ in~(3).   
 
Therefore $p$ carries $U$ homeomorphically onto a complete submanifold $U\subset L$ having only simply connected ends and cut off by the noncompact geodesic which is the projection of $\gamma$. All  the ends of $U$ are each fixed by $h$.  This implies that $\widetilde h$ fixes the corresponding ends of $U$ upstairs, hence that $\widehat h$ fixes $E\cap(a,b)$ pointwise. Indeed, $\widetilde h$ leaves invariant each component of $\partial\widetilde L\cap U$, hence $\widehat h$ fixes the ideal endpoints of these arcs. Since $L$ has no half planes, these endpoints are dense in $E\cap(a,b)$.  Thus, $\widehat h$ fixes $E$ pointwise and Theorem~\ref{isoisota} implies that $h$ is isotopic to the identity. Theorem~\ref{cc} is proven.
\end{proof}

For standard surfaces, Theorem~\ref{cc} is substantially stronger than Theorem~\ref{epsteinorig}.  Indeed,~(1) of Theorem~\ref{epsteinorig} implies~(1) of Theorem~\ref{cc}, but not conversely, and~(2) of Theorem~\ref{epsteinorig} implies~(2) of Theorem~\ref{cc}, but not conversely.  Part~(3) of Theorem~\ref{cc} shows that  Example~\ref{exfour} gives  the only possible examples of surfaces admiting homeomorphisms that are homotopic to the identity and fix all simply connected ends but are not isotopic to the identity.

\begin{corollary}
Let $L$ be a standard surface and $f,g:L\to L$ homotopic homeomorphisms.  If either condition $(1)$ or $(2)$ of \emph{Theorem~\ref{cc}} is satisfied or if $L$ is any surface except those in Example~\ref{exfour} and $f$ and $g$ agree on the set of simply connected ends, then $f$ is isotopic to $g$.
\end{corollary}

The proof is essentially the same as that of Corollary~\ref{epsteincor}.

The restriction in Theorem~\ref{cc}  to standard surfaces might cause concern.  As the following shows, this  restriction is very mild. 
 
 \begin{theorem}\label{nonstandard}
 Up to homeomorphism there are exactly $13$ nonstandard surfaces.  They are: the open disk, the closed disk, the open annulus, the half open annulus, the closed annulus, the open M\"obius band, the closed M\"obius band, the half plane $\mathbb{R}\times[0,\infty)$, the doubly infinite strip $[0,1]\times\mathbb{R}$, the sphere, the projective plane, the torus, the Klein bottle.
 \end{theorem}
 
 For these 13 surfaces, there are ad hoc proofs of Theorem~\ref{epsteinorig} (requiring that $h$ be orientation preserving in four cases) which  
can be gleaned from~\cite{Epstein:isotopy}.  Some are sufficiently elementary to be posed as exercises.

The next section will be devoted to the proof of Theorem~\ref{nonstandard}.

\section{The Nonstandard Surfaces}\label{13}

The following is quite elementary.

\begin{lemma}\label{unlucky}
The $13$ surfaces listed in \emph{Theorem~\ref{nonstandard} } are nonstandard.
\end{lemma}

For instance, the closed disk has a hyperbolic metric, but the boundary cannot be geodesic.  The open disk has a complete hyperbolic metric, necessarily isometric to the Poincar\'e metric.  This has imbedded half planes.  The open annulus has a complete hyperbolic metric and either one end is a cusp and the other a hyperbolic trumpet, or both ends are trumpets.  In any case there are imbedded half planes. The open M\"obius strip $M$ has the open annulus $A$ as its orientation cover.  The nontrivial deck transformation is an isometry of $A$ (relative to the lift of any  complete  hyperbolic metric on $M$) which interchanges the ends.  Thus both ends are trumpets and there is a half plane in $A$ which is carried to a disjoint half plane by the deck transformation. Thus $M$ contains a half plane.  Checking out the rest can be left to the reader.

We need the notion of Euler characteristic for possibly noncompact surfaces.  For any space  $N$ with finite dimensional real homology, the Euler characteristic is the alternating sum of the betti numbers $b_{i}(N)$.  For a closed, orientable surface $L$, this is $\chi(L)=2-b_{1}(L)$.  For every other surface, $\chi(L) = 1-b_{1}(L)$.  If $b_{1}(L)$ is infinite, we set $\chi(L)=-\infty$.  In particular, nonnegative Euler characteristic is finite.

Recall that, if $L$ is a  surface with $g$ handles, $c$ crosscaps, $b$ compact boundary components, $a$ annular ends (all finite integers) and no other boundary components or  ends, then the Euler characteristic of $L$ is  given by the following formula:
$$\chi(L) = 2 - 2g - c - b - a.$$
We fix the meaning of these letters.  Remark that the values of $g$ and $c$ are not individually well defined by $S$. For example, if $g=1=c$, we can also take $g=0$ and $c=3$ (cf.~\cite[p.~26, Lemma~7.1]{mass}). In fact, 3 crosscaps equals 1 handle and 1 crosscap. 

 \begin{remark}
 In particular, all of this applies to complete hyperbolic surfaces with finite area and compact geodesic boundary, where one also has the formula $$2\pi\chi(L)=-\area L,$$   computed by integrating the constant curvature $-1$.  In such a surface, the annular ends are cusps.
(See~\cite[pp. 31 - 37]{bca} for a discussion of such surfaces). 
 \end{remark}

\begin{definition}

A complete hyperbolic surface $S$ with geodesic boundary and finite area is called a \emph{generalized pair of pants} if $g=c=0$ and $b+a=3$. We will refer to a  cusp of a generalized pair of pants as a \emph{boundary component of length $0$}.

\end{definition}

The following  lemma is well known.

\begin{lemma}\label{wellknwn}

Given a triple of numbers $x_{i}\ge 0$, $1\le i\le 3$, there exists a generalized pair of pants whose  three boundary components have length $x_{i}$.

\end{lemma}

\begin{proposition}\label{finiteeuler}

Given a   surface $S$ with finite Euler characteristic $\chi<0$ and with boundary $b$ closed curves $\gamma_{1},\ldots,\gamma_{b}$ and endset consisting of $a\ge 0$  annular ends, and given positive real numbers $x_{1},\ldots,x_{b}$, there exists a complete hyperbolic metric on $S$ such that $\gamma_{i}$ is a geodesic of length $x_{i}$, $1\le i\le b$, and the annular ends are cusps. In particular, this metric is standard.

\end{proposition}

\begin{proof}

Cut the handles of $S$ apart along $g$ curves and the crosscaps apart along $c$ curves to yield a surface $S'$ homeomorphic to a sphere with $2g+c+b$ boundary components and $a$ punctures.  Since $\chi = 2 - 2g - c - b - a<0$ it follows that $2<2g+c+b+a$. Thus $2g+c+b+a\ge3$ and $S'$ has a generalized pair of pants decomposition. By Lemma~\ref{wellknwn}, each generalized pair of pants can be given a hyperbolic metric so that the annular ends are cusps and the boundary lengths are such that the regluing can be done to give a hyperbolic metric on $S$ with the $\gamma_{i}$ having the correct lengths.
\end{proof}

The following lemma is classical. See Casson-Bleiler~\cite[p. 37]{bca}.

\begin{lemma}\label{MV}

If $L$ is a surface with finite Euler characteristic and $r<\infty$ noncompact boundary components then,
$
\chi(2L) = 2\chi(L) + r.
$

\end{lemma}

\begin{proposition}\label{unlucky13}
If $L$ is a nonstandard surface with finite Euler characteristic and finitely many boundary components, then it is homeomorphic to one of the $13$ surfaces in \emph{Lemma~\ref{unlucky}}.
\end{proposition}

\begin{proof}
Let $L$, as above,  have compact boundary and finite Euler characteristic.  If $\chi(L) \ge 0$, then $2g+c+b+a\le 2$. The possibilities are: 

\begin{enumerate}

\item $g\le 1$ and $c=b=a=0$, so $L$ is the sphere or torus.

\item $0<c\le 2$ and $g=b=a=0$, so $L$ is the  projective plane or Klein bottle.

\item $0<b\le2$ and $g=c=a=0$, so $L$ is the closed disk or closed annulus.

\item $0<a\le2$ and $g=c=b=0$, so $L$ is the open disk or open annulus.

\item $c=b=1$ and $g=a=0$, so $L$ is the closed M\"obius strip.

\item $c=a=1$ and $g=b=0$, so $L$ is the open M\"obius strip.

\item $b=a=1$ and $g=c=0$, so $L$ is the half open annulus.

\end{enumerate}
Thus, any surface $L$ with finite Euler characteristic and only compact boundary components, other than the $11$ surfaces just listed, has negative Euler characteristic and so, by Proposition~\ref{finiteeuler},  has a standard hyperbolic metric.  By Lemma~\ref{unlucky}, the above 11 are nonstandard.

Let $L$ have finite Euler characteristic and finitely many boundary components, not all of which are  compact.  Since $L$ is nonstandard, so is $2L$ (Lemma~\ref{2Lstand}).  By Lemma~\ref{MV}, $2L$ has finite Euler characteristic.  Since the boundary is empty, $2L$ must be in the above list of $11$ surfaces.   Since $L$ cannot be compact and $2L$ cannot have boundary, $2L$ is either the open disk, the open annulus or the open M\"obius strip. The only surface whose double is the open disk is the half plane. The doubly infinite strip is the only surface with noncompact boundary  whose double is the open annulus, and no surface has double equal to the open M\"obius strip. This completes the proof. 
\end{proof}

\begin{lemma}\label{therest}

If $L$ is a   surface with infinitely many boundary components, then $2L$ has at least one end that is not annular.

\end{lemma}

\begin{proof}

Suppose   that $L$ is a surface  such that every end of $2L$ is  annular. Then $2L$ has finitely many annular ends and there exists a compact surface $K\subset L$ such that the finitely many components of $2L\setminus 2K$ are open annuli. Thus, every one of the finitely many  components of $L\setminus K$ is either of the form $S_{1}\times (0,\infty)$ or $[0,1]\times(0,\infty)$. Thus, $L$ has  finitely many boundary components. The lemma now follows by contradiction.
\end{proof}

By Proposition~\ref{unlucky13}, the following completes the proof of Theorem~\ref{nonstandard}.

\begin{proposition}
If $L$ is a surface with infinitely many boundary components and/or infinite Euler characteristic, then $L$ is standard.
\end{proposition}

\begin{proof}
If $L$ has nonempty boundary, then $L$ is standard if and only if $2L$ is standard (Lemma~\ref{2Lstand}).  Also, under our hypothesis,  $2L$ has   at least one end that is not annular (Lemma~\ref{therest}). Hence we can limit our attention to surfaces $L$ without boundary with at least one end that is not annular.

Thus, there is an  exhaustion $K_{0}\subset K_{1}\subset\cdots\subset K_{i}\subset\cdots$ of $L$ by connected surfaces $K_{i}$ with finite Euler characteristic such that  the finitely many connected components  of $L\smallsetminus K_{i}$ are neighborhoods of non-annular ends,  the frontier of each being a simple closed curve. Let the collection of all of these simple closed curves be denoted by $\mathfrak{A}=\{\alpha_{i}\}_{i=1}^{\infty}$.  Note that $\mathfrak{A}$ partitions $2L$ into a countable family $\{B_{k}\}$  of  submanifolds with finite Euler characteristic and with boundary composed of elements of $\mathfrak{A}$.  We can assume the exhaustion is chosen so that all the $B_{k}$ have negative Euler characteristic.  Then by Proposition~\ref{finiteeuler}, each $B_{k}$ has a standard hyperbolic metric such that every simple closed curve in $\mathfrak{A}$ is a geodesic of length $1$. Thus, these hyperbolic metrics can be fitted together to give a hyperbolic metric on $L$. Evidently, an imbedded half plane $H\subset L$ would have to meet  closed geodesics in $\partial B_{k}$, for infinitely many indices $k$.  Since closed geodesics are not nullhomotopic, the intersections of these geodesics with $H$ will be arcs which, together with an arc in $\partial H$ would form a geodesic digon.  Such digons are forbidden in hyperbolic geometry.
\end{proof}

\end{document}